\documentclass[preprint,12pt]{elsarticle}

\usepackage{amssymb}
\usepackage{amsmath}
\usepackage{amsthm}

\usepackage{thmtools, thm-restate}

\usepackage{algorithm}
\usepackage{algpseudocode}

\usepackage{tikz}
\usetikzlibrary{arrows.meta}
\usetikzlibrary{shapes.misc}

\usepackage{nicematrix}

\usepackage[colorlinks=true]{hyperref}
\usepackage{cleveref}

\crefname{appendix}{}{}

\newcommand{\C}{\mathbb{C}}

\theoremstyle{Theorem}
\newtheorem{thm}{Theorem}
\newtheorem{prop}{Proposition}

\theoremstyle{definition}

\usepackage{lineno}

\journal{Linear Algebra and its Applications}

\begin{document}

\begin{frontmatter}



\title{A Hidden Variable Resultant Method for the Polynomial Multiparameter Eigenvalue Problem}


\author[label1]{Emil Graf} 
\author[label1]{Alex Townsend} 

\affiliation[label1]{organization={Cornell University},
            addressline={212 Garden Ave}, 
            city={Ithaca},
            postcode={14853}, 
            state={NY},
            country={USA}}

\begin{abstract}
We present a novel, global algorithm for solving polynomial multiparameter eigenvalue problems (PMEPs) by leveraging a hidden variable tensor Dixon resultant framework.  Our method transforms a PMEP into one or more univariate polynomial eigenvalue problems, which are solved as generalized eigenvalue problems.  Our general approach avoids the need for custom linearizations of PMEPs.  We provide rigorous theoretical guarantees for generic PMEPs and give practical strategies for nongeneric systems.  Benchmarking on applications from aeroelastic flutter and leaky wave propagation confirms that our algorithm attains high accuracy and robustness while being broadly applicable to many PMEPs. 
\end{abstract}




\end{frontmatter}



\section{Introduction}
\label{sec:intro}

Polynomial multiparameter eigenvalue problems (PMEPs) appear in a variety of applications, including the study of aeroelastic flutter \cite{pons2017aef,pons2018aef}, analysis of delay-differential equations \cite{jarlebring2009dde}, computation of the signed distance between ellipsoids \cite{iwata2015signeddist}, computation of zero-group-velocity points in waveguides \cite{kiefer2023waves}, and computation of properties of leaky waves \cite{gravenkamp2025leakywaves}. A PMEP takes the form
\begin{equation} \label{eq:MPEigForm}
P_i(x_1,\ldots,x_d) \mathbf{v}_i = 0, \quad 1 \leq i \leq d,
\end{equation}
where $P_i(x_1,\ldots,x_d) \in \C^{n_i \times n_i}[x_1,\ldots,x_d]$ for integers $n_1,\ldots,n_d$ is a multivariate matrix polynomial, meaning $P_i(x_1^*,\ldots,x_d^*) \in \C^{n_i \times n_i}$ for any $(x_1^*,\ldots,x_d^*) \in \C^d$. The task is to find eigenvalues $(x_1^*,\ldots,x_d^*) \in \C^d$ and eigenvectors $\mathbf{v}_i^* \in \C^{n_i}, 1 \leq i \leq d,$ that solve all equations simultaneously. PMEPs can be viewed as a generalization of multivariate polynomial rootfinding problems, with matrix instead of scalar coefficients, or multiparameter eigenvalue problems, where we replace linear matrix polynomials with nonlinear ones.

Despite the prevalence of applications, we believe there is no general-purpose global method to solve PMEPs. As the PMEP is a generalization of polynomial rootfinding and multiparameter eigenvalue problems, a global PMEP solver must contend with the challenges of both problems. As we observe in \cref{sec:background}, methods for polynomial rootfinding problems must address the nonlinearity of the starting problem, and methods for multiparameter eigenvalue problems face the obstacle of distinct eigenvectors for each equation, which makes it challenging to solve all equations simultaneously; a PMEP solver must simultaneously address both issues, making construction of such a method difficult.

In the absence of global methods, researchers have constructed custom linearizations tailored to specific degree constraints and sparsity patterns of the matrix polynomials $P_i$. For example, the quadratic two parameter eigenvalue problem, in which $d=2$ and the total degree of each $P_i$ is $\leq 2$, has received particular attention \cite{bor2010quadeig,hochstenbach2012quadeig}. While these custom methods are often very successful at solving the problems they are constructed for, the work of developing a custom method for each instance of \cref{eq:MPEigForm} discourages scientists from approaching applications that involve solving more complicated PMEPs.

With this need in mind, we derive a method to solve any PMEP globally. Our method reduces a PMEP to one or more polynomial eigenvalue problems (PEPs), which can be solved by known linearizations, such as companion or colleague \cite{mackey2006vs,nakatsukasa2016stability}, and the QZ algorithm.  Our method is competitive with existing methods based on case-by-case linearizations for several applications. More importantly, it presents an opportunity for practitioners to solve relevant PMEPs immediately without the time-consuming construction of custom methods.

\subsection{The Generic $d$-degree PMEP}
\label{subsec:gen}

For theoretical analysis later, we consider a particular structure of PMEP. A (matrix) polynomial is called generic $d$-degree if it is of the form
\begin{equation} \label{eq:maxdegsystem}
	P_i(x_1,\ldots,x_d) = \sum_{i_1 = 0}^{\tau_1} \cdots \sum_{i_d = 0}^{\tau_d} P_{i_1,\cdots,i_d} x_1^{i_1} \cdots x_d^{i_d},
\end{equation}
where $\tau_1,\ldots,\tau_d$ are positive integers, and $P_{i_1,\ldots,i_d}$ is an $n_i \times n_i$ matrix with entries that are distinct indeterminates. A PMEP is called generic $d$-degree, or more specifically generic $d$-degree $\tau_1,\ldots,\tau_d$, if every matrix polynomial is generic $d$-degree of the same multi-degree $\tau_1,\ldots,\tau_d$. 

The practical consequence of analysis for generic $d$-degree systems is that a system of the form in \cref{eq:maxdegsystem} with random coefficients $P_{i_1,\ldots,i_d} \in \C^{n_i \times n_i}$ can be expected to behave like a generic system with probability one. While generic systems with indeterminate coefficients are impractical, the theoretical analysis can be useful for practical systems that are close enough to random to share many of the same properties that we can prove for generic systems. 

Even if the PMEP is not constructed with indeterminate or random coefficients, any matrix polynomial can be written in the form in \cref{eq:maxdegsystem}, which is the most natural form for the construction in our algorithm, even if some or many of the matrix coefficients are zero. We refer to a system as a maximal $d$-degree $\tau_1,\ldots,\tau_d$ system if it is represented as in \cref{eq:maxdegsystem}, but the matrix coefficients $P_{i_1,\cdots,i_d} \in \C^{n_i \times n_i}$ are not necessarily indeterminates. We represent all our systems as maximal degree systems for some multidegree $\tau_1,\ldots,\tau_d$, though not all are generic.

\subsection{Structure of the Paper}

We begin by giving background on the pre-existing methods that inspire our approach (see \cref{sec:background}). Then we outline our method to solve generic PMEPs (see \cref{sec:outline}) and prove that the method works for generic PMEPs (see \cref{sec:theory}). We then examine what can go wrong in the nongeneric case, and explore the numerical techniques we use to overcome challenging examples (see \cref{sec:finetuning}). We finish with some practical experiments based on applications in \cite{pons2017aef,pons2018aef,gravenkamp2025leakywaves}  to highlight the utility of our method (see \cref{sec:num}).

\section{Background}
\label{sec:background}

Our algorithm combines ideas from the operator determinants method for the linear multiparameter eigenvalue problem and hidden variable resultant methods for systems of polynomial equations. As seen in table \ref{tab:evp}, the PMEP is a generalization of both problems.  

%

\subsection{The Linear Multiparameter Eigenvalue Problem}
\label{subsec:opdet}

 In the special case where all of the matrix polynomials $P_i$ are linear, the PMEP in \cref{eq:MPEigForm} reduces to a linear multiparameter eigenvalue problem (MEP) of the form:
\begin{equation} \label{eq:MEP}
W_i(\mathbf{x}) \mathbf{v}_i = \left(V_{i0} - \sum_{j=1}^d x_j V_{ij} \right) \mathbf{v}_i = 0, \quad 1 \leq i \leq d,
\end{equation}
with $V_{ij} \in \mathbb{C}^{n_i \times n_i}$ for integers $n_1,\ldots,n_d$. A solution of the multiparameter eigenvalue problem consists of an eigenvalue $(x_1^*,\ldots,x_d^*) \in \C^d$ and eigenvectors $\mathbf{v}_i^* \in \C^{n_i}, 1 \leq i \leq d$, that satisfy \cref{eq:MEP}. MEPs are extremely well-studied, and a popular global method exists to solve them. The usual approach to solve an MEP is via operator determinants \cite{atkinson1972multieig}, where one constructs
\[
\Delta_0 =
\begin{vmatrix}
    V_{11}& V_{12} & \cdots & V_{1d} \\
    V_{21}& V_{22} & \cdots & V_{2d} \\
    \vdots & \vdots & \ddots & \vdots \\
    V_{d1} & V_{d2} & \cdots & V_{dd}
\end{vmatrix}_{\otimes}, \,
\Delta_i = 
\begin{vmatrix}
    V_{11}& \cdots & V_{1,i-1} & V_{10} & V_{1,i+1} & \cdots & V_{1d} \\
    V_{21}& \cdots & V_{2,i-1} & V_{20} & V_{2,i+1} & \cdots & V_{2d} \\
    \vdots & \ddots & \vdots & \vdots & \vdots & \ddots & \vdots \\
    V_{d1} & \cdots & V_{d,i-1} & V_{d0} & V_{d,i+1} & \cdots & V_{dd}
\end{vmatrix}_{\otimes},
\]
where the notation $\begin{vmatrix} M \end{vmatrix}_{\otimes}$ denotes taking the block determinant of $M$ with multiplication replaced by Kronecker products. That is, for example,  given by 
\begin{equation} \label{eq:leibniz}
\Delta_0 = \sum_{\sigma \in S_d} \textbf{sgn}(\sigma) V_{1,\sigma(1)} \otimes V_{2,\sigma(2)} \otimes \cdots \otimes V_{d,\sigma(d)},
\end{equation}
where $S_d$ is the group of permutations on $\{1,\ldots,d\}$ and $\textbf{sgn}(\sigma)$ is the sign of the permutation $\sigma$. As $\otimes$ is not commutative, the determinant should be expanded as in \cref{eq:leibniz} using the Leibniz formula and the Kronecker products multiplied from the first row to the last row.

After constructing the operator determinants, one solves the GEPs given by 
\begin{equation} \label{eq:opdetGEP}
(\Delta_i -x_i \Delta_0)\mathbf{z}_i = 0, \quad 1 \leq i \leq d.
\end{equation}
It turns out that the eigenvalues $x_i$ of the GEPs are the $i$th coordinates of the eigenvalues of the original MEP and that the eigenvectors $\mathbf{z}_i$ are all equal and are given by the Kronecker product $\mathbf{v}_1^* \otimes \cdots \otimes \mathbf{v}_d^*$ \cite{atkinson1972multieig}. All global methods for MEPs of which we are aware are based on this idea.

We will use operator determinants as a key building block in our PMEP solver. The crucial innovation of this method is that the Kronecker products allow us to combine many separate eigenvalue problems into one; in particular, the original problems have completely distinct eigenvectors, but this construction allows each GEP in \cref{eq:opdetGEP} to have a single eigenvector that is the Kronecker product of the original eigenvectors. Our PMEP solver will incorporate Kronecker structure for the same reason.

\subsection{Hidden Variable Resultants}
\label{subsec:hvr}

In computational algebraic geometry, researchers study multivariate rootfinding problems of the form
 \begin{equation} \label{eq:polysystem}
p_i(x_1,\ldots,x_d) = 0, \quad 1 \leq i \leq d,
\end{equation}
where $p_i \in \C[x_1,\ldots,x_d]$ is a multivariate polynomial. Hidden variable resultant methods are one popular method for solving such systems; they are particularly useful as they convert a polynomial system to a GEP, which allows practitioners to use well studied algorithms from numerical linear algebra such as the QZ algorithm to do much of the work in a solver. We will generalize this idea to convert a PMEP to a GEP, with the same end goal of using QZ as a significant step of our algorithm.

A hidden variable resultant method views each polynomial $p_i$ as a polynomial in the variables $x_1,\ldots,x_{d-1}$ with coefficients that are functions of $x_d$. That is, we write
\begin{equation}
	p_i(x_1,\ldots,x_d) = \sum_{i_1 = 0}^{\tau_1} \cdots \sum_{i_{d-1} = 0}^{\tau_{d-1}} \left( p_{i_1,\ldots,i_{d-1}}(x_d) \right) x_1^{i_1}  \cdots x_{d-1}^{i_{d-1}},
\end{equation}
with $p_{i_1,\ldots,i_{d-1}}(x_d) \in \C[x_d]$ and $\tau_1,\ldots,\tau_{d-1}$ as in \cref{eq:maxdegsystem}. We have now converted the problem from one in which we need to find the coordinates of solutions to one in which we ask whether there exists a solution. Specifically, we have $d$ polynomials $p_1,\ldots,p_d$ in $d-1$ variables, and we search for a function that can tell us when these polynomials have a common root. The resultant is a polynomial function of the coefficients of $p_1,\ldots,p_d$ that vanishes if and only if $p_1,\ldots,p_d$ have a common root \cite[Chapt. 3]{cox2005ag}. We can use the resultant to find the values of $x_d$ for which the system has a possible solution.

To transform into an eigenvalue problem, a matrix representation of the resultant is used; that is, we take a matrix $R(p_1,\ldots,p_d)$ whose entries are polynomial functions of the coefficients of $p_1,\ldots,p_d$ and whose determinant is the resultant. Thus, $R$ is singular if and only if the polynomials $p_1,\ldots,p_d$ have a common root. Crucially, this moves the problem from algebraic geometry into linear algebra, where we have more powerful numerical tools.

Since we have hidden the variable $x_d$, the matrix representation of the resultant is a matrix polynomial $R(x_d)$ that is singular at the values of $x_d$ for which $p_1,\ldots,p_d$ have a common root in the variables $x_1,\ldots,x_{d-1}$. This is a polynomial eigenvalue problem whose solutions give all the possible $x_d$ coordinates of the roots. Once the $x_d$ coordinates are found, the other coordinates can be found by reducing the problem and repeating the method, or, for some resultants, by extracting solutions from the eigenvectors of $R(x_d)$. 

In computational algebraic geometry, a choice of resultant, such as the Dixon/Cayley resultant, leads to a well-studied method \cite{noferini2016instability,nakatsukasa2015bezout} for finding all the roots of multivariate polynomial systems. We use a hidden variable resultant method as another building block in our PMEP solver. In particular, just as the polynomial resultant is an effective way to convert a nonlinear rootfinding problem into a linear eigenvalue problem, our resultant generalization is an effective way to construct a linear problem from an originally nonlinear PMEP, giving a GEP that can be handled by a straightforward use of the QZ algorithm.

\subsection{The Polynomial Hidden Variable Dixon Resultant}

Many resultant choices exist for solving polynomial systems.  Our method for PMEPs generalizes the Dixon resultant, originally given for three polynomials in two variables in \cite{dixon1908res}, though it can be constructed for any number of polynomials~\cite{noferini2016instability}.\footnote{The Dixon resultant is sometimes referred to as the Cayley resultant, and in two dimensions is often called the B\'{e}zout resultant.}
Given the polynomial system in \cref{eq:polysystem}, with each $p_i$ of maximal $d$-degree $\tau_1,\ldots,\tau_d$ as in \cref{eq:maxdegsystem}, define
\begin{align*}
&f_{\text{Dixon}}(s_1,\ldots,s_{d-1},t_1,\ldots,t_{d-1},x_d) = \\ & \hspace{.3cm} \frac{\begin{vmatrix}
    p_1(s_1,s_2,\ldots,s_{d-1},x_d) & p_1(t_1,s_2,\ldots,s_{d-1},x_d) & \cdots & p_1(t_1,t_2,\ldots,t_{d-1},x_d) \\
    p_2(s_1,s_2,\ldots,s_{d-1},x_d) & p_2(t_1,s_2,\ldots,s_{d-1},x_d) & \cdots & p_2(t_1,t_2,\ldots,t_{d-1},x_d) \\
    \vdots & \vdots & \ddots & \vdots \\
    p_d(s_1,s_2,\ldots,s_{d-1},x_d) & p_d(t_1,s_2,\ldots,s_{d-1},x_d) & \cdots & p_d(t_1,t_2,\ldots,t_{d-1},x_d) \\
\end{vmatrix}}{\prod_{i=1}^{d-1}(s_i-t_i)},
\end{align*}
where the vertical bars denote taking the matrix determinant. It is clear from examining the numerator that the degree of $f_{\text{Dixon}}$ in $s_i$ is bounded above by $\alpha_i =  i\tau_i - 1$ and the degree in $t_i$ is bounded by $\beta_i = (d-i)\tau_i-1$. The degree in $x_d$ is bounded by $d \tau_d$.
We can expand $f_{\text{Dixon}}$ in the monomial basis as
\begin{equation} \label{eq:polyDixonexpansion}
f_{\text{Dixon}} = \sum_{i_1=0}^{\alpha_1} \cdots \sum_{i_{d-1}=0}^{\alpha_{d-1}} \sum_{j_1=0}^{\beta_1} \cdots \sum_{j_{d-1}=0}^{\beta_{d-1}} a_{i_1, \ldots, i_{d-1}, j_1, \ldots, j_{d-1}}(x_d) \prod_{k=1}^{d-1} s_k^{i_k} \prod_{k=1}^{d-1} t_k^{j_k},
\end{equation}
where $a_{i_1, \ldots, i_{d-1}, j_1, \ldots, j_{d-1}}(x_d)$ is a polynomial function of $x_d$.
We then define the hidden variable Dixon resultant $R(x_d)$ to be the $ \left( (d-1)! \prod_{k=1}^{d-1} \tau_k \right) \times \left( (d-1)! \prod_{k=1}^{d-1} \tau_k \right) $ matrix polynomial function of $x_d$ given by unfolding the expansion in \cref{eq:polyDixonexpansion}. In particular, we unfold so that the columns of $R(x_d)$ are indexed by the $s$ monomials and the rows by the $t$ monomials. 

The Dixon resultant has the key property of resultants from \cref{subsec:hvr} that allows the construction of a hidden variable resultant method; $R(x_d)$ has eigenvalues $x_d^*$ that are the $d$th coordinate of each solution to \cref{eq:polysystem} \cite{noferini2016instability}. Better still, the $x_1,\ldots,x_{d-1}$ coordinates can be extracted from the eigenvectors of $R(x_d)$. Thus, this construction can be used to create a polynomial hidden variable resultant solver.

\subsubsection{Multivariate Block Vandermonde Vectors}

In \cite{noferini2016instability}, it is proved that the eigenvectors of the Dixon resultant have multivariate Vandermonde structure. As mentioned in \cref{subsec:opdet}, the eigenvectors of the GEPs resulting from the operator determinants method have a block/tensor structure derived from the original eigenvectors \cite{atkinson1972multieig}. Our PMEP method will have similarly structured eigenvectors.  

For ease of notation,  suppose $U$ is a tensor with $U_{i_1,\ldots,i_d,\ell} = \prod_{k=1}^d x_k^{i_k} \mathbf{v}_{(\ell)}, \ell = 1,\ldots,r,$ with $\mathbf{v} \in \C^r$ for some integer $r$, and $\mathbf{v}_{(\ell)}$ the $\ell$th entry of $\mathbf{v}$. Then, we say that the vector $V = \text{vec}(U)$ is a multivariate block Vandermonde vector.  If $r=1$, then $V$ is a multivariate Vandermonde vector without blocks. If $d=1$ and $r=1$, then $V$ is a standard univariate Vandermonde vector.

\begin{table}[]
    \centering
    \caption{Methods for multivariate linear equations,  rootfinding,  and eigenproblems. Moving up or to the left is both a specialization of PMEPs and of our method. We make this connection concrete in \cref{app:conn}.}
    \resizebox{\columnwidth}{!}{
    \begin{tabular}{ccc}
    \\
    \hline
    \begin{tabular}{c}
    Linear System \\
    $\sum_{j=1}^d a_{ij} x_j = b_i, 1 \leq i \leq d$ \\
    Solved by Cramer's rule
    \end{tabular} & &
    \begin{tabular}{c}
    Multivariate Eigenproblem \\
    $W_i(\mathbf{x}) \mathbf{v}_i = 0,1 \leq i \leq d$ \\
    Solved by operator determinants \cite{atkinson1972multieig}
    \end{tabular} \\ \\
    \begin{tabular}{c}
    Multivariate Polynomial Rootfinding \\
    $p_i(\mathbf{x}) = 0,1 \leq i \leq d$ \\
    Solved by hidden variable Dixon resultant \cite{noferini2016instability}
    \end{tabular} & & 
    \begin{tabular}{c}
    Multivariate Polynomial Eigenproblem \\
    $P_i(\mathbf{x})\mathbf{v}_i = 0,1 \leq i \leq d$ \\
    Solved by tensor Dixon resultant
    \end{tabular} \\
    \hline
	\end{tabular}}
    \label{tab:evp}
\end{table}

\subsection{Developing a Resultant for Matrix Polynomials}
\label{subsec:dev}


Given a PMEP with $d$ equations and $d$ variables in~\cref{eq:MPEigForm},  we want to develop a hidden variable resultant method. With this goal, we first hide the variable $x_d$ so that we have a PMEP with $d$ equations and $d-1$ variables. Then we need to find a matrix function that is singular if and only if the matrix polynomials $P_i$ have a common eigenvalue.  

A first attempt at constructing such a matrix function might be a straightforward generalization of the scalar Dixon resultant.  Unfortunately,  this does not work.   To see this,  consider the following two univariate matrix polynomials from~\cite{gohberg1976resultants}:
\begin{equation}
A(\lambda) = \begin{pmatrix}
\lambda - 1 & 0 \\
1 & \lambda -1
\end{pmatrix}, \quad 
B(\lambda) = \begin{pmatrix}
\lambda  & 1 \\
0 & \lambda -2
\end{pmatrix}.
\label{eq:Example}
\end{equation}
We want to construct a matrix that is singular if and only if the two matrix polynomials have a common eigenvalue. The straightforward generalization of the two-dimensional Dixon resultant (known as the B\'{e}zout resultant),  involves replacing scalar multiplication with matrix multiplication in its construction.  This approach gives us the following attempt at a generalization: 
\[
f_{\text{Dixon}}(s,t) = \frac{\begin{vmatrix}
    A(s) & A(t) \\
   B(s)  & B(t) \\
\end{vmatrix}}{(s-t)} = \frac{A(s)B(t)-A(t) B(s)}{(s-t)} = \begin{pmatrix} 1& 1 \\ -1 & -1 \end{pmatrix}.
\]
Because the Dixon function is constant in $s$ and $t$, the resultant is equal to the Dixon function, which is singular, even though the matrix polynomials do not have a common eigenvalue. Thus, this straightforward generalization of the Dixon resultant to the matrix case is inadequate.  

This approach fails because it does not account for the fact that the matrix polynomials can have different eigenvectors but common eigenvalues. It is necessary to take some inspiration from the operator determinants method of \cref{subsec:opdet} to properly account for the distinct eigenvectors of each matrix polynomial.

In \cite{heinig1977bezoutiant}, another possible generalization of the scalar Dixon resultant is given in two dimensions as
\begin{equation} \label{eq:bez}
f_{\text{Dixon}}(s,t) = \frac{A(s) \otimes B(t)-A(t) \otimes B(s)}{(s-t)}.
\end{equation}
It is proved that $f_{\text{Dixon}}(s,t)$ is a resultant for bivariate matrix polynomials, i.e., the matrix it generates after unfolding the analogous expansion in \cref{eq:polyDixonexpansion} has the property that it is singular if and only if the two matrix polynomials have a common eigenvalue. This is a successful approach precisely because the Kronecker products correctly tie the eigenvectors of the original matrix polynomials together; this construction has eigenvectors that are related to Kronecker products of the original eigenvectors of $A(\lambda),B(\lambda)$.

For our example in~\cref{eq:Example},  the resultant is still equal to the Dixon function, which is now
\begin{equation}
f_{\text{Dixon}}(s,t) =  \frac{A(s) \otimes B(t)-A(t) \otimes B(s)}{(s-t)} = \begin{pmatrix} 1& 1 & 0 & 0 \\ 0 & -1 & 0 & 0 \\ -1 & 0 & 1 & 1 \\ 0 & -1 & 0 & -1 \end{pmatrix},
\end{equation}
which is nonsingular, showing that the matrix polynomials $A$ and $B$ do not have a common eigenvalue.

The construction in \cref{eq:bez} is studied further in \cite{barnett1980bezoutian,lerer1982bezoutian}, while \cite{gohberg1976resultants} gives a similar definition based on the Sylvester resultant. However,  these papers do not study the possibility of developing a hidden variable resultant method. It is suggested in \cite{iwata2015signeddist} that a hidden variable resultant method based on the definition in \cref{eq:bez} is possible, but the authors instead use the theory to prove the validity of a customized approach for solving a particular PMEP related to computing signed distances between ellipsoids. We were motivated by the suggestion in~\cite{iwata2015signeddist} to fully develop a hidden variable resultant method for PMEPs in any dimension. 


\subsection{Linearizations of Polynomial Eigenvalue Problems}
\label{subsec:lin}

Similarly to the polynomial resultant, our generalization of the hidden variable Dixon resultant constructs a univariate matrix polynomial $R(x_d)$ whose eigenvalues are the $d$th coordinates of solutions to the PMEP. The benefit of this is that we can lean on numerically robust methods to solve the resulting polynomial eigenvalue problem.

A polynomial eigenvalue problem (PEP)  is an equation $P(\lambda) \mathbf{v} = 0$, where $P(\lambda) = \sum_{i=0}^m P_i \lambda^i$ is a univariate matrix polynomial with $P_i \in \C^{n \times n}$, and the task is to find eigenvalue-eigenvector pairs $\lambda \in \C,\mathbf{v} \in \C^n$ that solve the equation. There are various well-studied linearizations such as companion or colleague  \cite{mackey2006vs,nakatsukasa2016stability} that construct a linear matrix polynomial $A-\lambda B$ with the same eigenvalues as $P(\lambda)$. In the monomial basis, we use the companion linearization, which is given by
$$
\begin{bmatrix}
P_m & 0 & \cdots & 0 \\
0 & I_n & \ddots & \vdots \\
\vdots & \ddots & \ddots & 0 \\
0 & \cdots & 0 & I_n
\end{bmatrix} \lambda + 
\begin{bmatrix}
P_{m-1} & P_{m-2} & \cdots & P_0 \\
-I_n & 0 & \cdots & 0 \\
\vdots & \ddots & \ddots & \vdots \\
0 & \cdots & -I_n  & 0
\end{bmatrix}.
$$

The key property of the companion linearization is that it has the same eigenvalues as $P(\lambda)$. In addition, the eigenvector structure of the linearization is known \cite{mackey2006vs}, so we can extract the eigenvectors of $P(\lambda)$ from the eigenvectors of the linearization. This is important for our method, as the eigenvectors give the $x_1,\ldots,x_{d-1}$ coordinates of the solutions. We use a companion-like linearization to solve the PEP that our PMEP solver constructs.  Once linearized, we numerically solve the GEP with the QZ algorithm.

\section{A Generic PMEP Solver}
\label{sec:outline}

Our PMEP solver is built from a combination of a hidden variable resultant method (see~\cref{subsec:hvr}) and operator determinants (see~\cref{subsec:opdet}), which gives us a multivariate resultant method for matrix polynomials.  We use this to construct a PEP that gives one of the coordinates of the solutions to the PMEP in \cref{eq:MPEigForm}. The remaining coordinates can be extracted from the eigenvectors of the resultant PEP.

\subsection{The Hidden Variable Tensor Dixon Resultant}

Given a PMEP of the form
\begin{equation*}
P_i(x_1,\ldots,x_d) \mathbf{v}_i = 0, \quad 1 \leq i \leq d,
\end{equation*}
with each $P_i$ being of maximal $d$-degree $\tau_1,\ldots,\tau_d$ (see~\cref{eq:maxdegsystem}),  our hidden variable tensor Dixon resultant is constructed as follows.\footnote{We outline every aspect of the algorithm in the monomial basis, as it is algebraically neat, but one can develop the algorithm in any degree graded basis. In particular, we implement the algorithm in the Chebyshev basis.}
Define
\begin{align*}
&f_{\text{Dixon}}(s_1,\ldots,s_{d-1},t_1,\ldots,t_{d-1},x_d) = \\ & \frac{\begin{vmatrix}
    P_1(s_1,s_2,\ldots,s_{d-1},x_d) & P_1(t_1,s_2,\ldots,s_{d-1},x_d) & \cdots & P_1(t_1,t_2,\ldots,t_{d-1},x_d) \\
    P_2(s_1,s_2,\ldots,s_{d-1},x_d) & P_2(t_1,s_2,\ldots,s_{d-1},x_d) & \cdots & P_2(t_1,t_2,\ldots,t_{d-1},x_d) \\
    \vdots & \vdots & \ddots & \vdots \\
    P_d(s_1,s_2,\ldots,s_{d-1},x_d) & P_d(t_1,s_2,\ldots,s_{d-1},x_d) & \cdots & P_d(t_1,t_2,\ldots,t_{d-1},x_d) \\
\end{vmatrix}_{\otimes}}{\prod_{i=1}^{d-1}(s_i-t_i)},
\end{align*}
where the notation $\begin{vmatrix} M \end{vmatrix}_{\otimes}$ denotes taking the block determinant with multiplication replaced by Kronecker products (see~\cref{eq:leibniz}).  We expand this using the Leibniz formula in terms of products in the order $P_1 \otimes P_2 \otimes \cdots \otimes P_d$. As in the scalar case, the degree of $f_{\text{Dixon}}$ in $s_i$ is bounded above by $\alpha_i =  i\tau_i - 1$, the degree in $t_i$ is bounded by $\beta_i = (d-i)\tau_i-1$, and the degree in $x_d$ is bounded by $d \tau_d$. Therefore, we find that 
\begin{equation} \label{eq:Dixonexpansion}
f_{\text{Dixon}} = \sum_{i_1=0}^{\alpha_1} \cdots \sum_{i_{d-1}=0}^{\alpha_{d-1}} \sum_{j_1=0}^{\beta_1} \cdots \sum_{j_{d-1}=0}^{\beta_{d-1}} A_{i_1, \ldots, i_{d-1}, j_1, \ldots, j_{d-1}}(x_d) \prod_{k=1}^{d-1} s_k^{i_k} \prod_{k=1}^{d-1} t_k^{j_k},
\end{equation}
where $A_{i_1, \ldots, i_{d-1}, j_1, \ldots, j_{d-1}}(x_d)$ is a $\left( \prod_{i=1}^d n_i \right) \times \left( \prod_{i=1}^d n_i \right)$ matrix polynomial function of $x_d$.
We then define the hidden variable tensor Dixon resultant $R(x_d)$ to be the $ \left( \prod_{i=1}^d n_i \cdot (d-1)! \prod_{k=1}^{d-1} \tau_k \right) \times \left( \prod_{i=1}^d n_i \cdot (d-1)! \prod_{k=1}^{d-1} \tau_k \right) $ matrix polynomial function of $x_d$ given by unfolding the expansion in \cref{eq:Dixonexpansion}. In particular, we unfold so that the block columns of $R(x_d)$ of width $\prod_{i=1}^d n_i$ are indexed by the $s$ monomials and the block rows of height $\prod_{i=1}^d n_i$ are indexed by the $t$ monomials.

Now we have a PEP whose eigenvalues give the $x_d$ coordinates of solutions to the original PMEP. We can solve this with known linearizations and the QZ algorithm as described in \cref{subsec:lin}. Moreover, we can actually extract the $x_1,\ldots,x_{d-1}$ coordinates of the solutions from this PEP as well.
In \cref{sec:theory}, we prove that the eigenvector of $R(x_d^*)$ corresponding to an eigenvalue $x_d^*$ is the multivariate block Vandermonde vector $\text{vec}(U)$, where $U$ is a tensor with
$$
U_{i_1,\ldots,i_{d-1},\ell} = \prod_{k=1}^{d-1} (x_k^*)^{i_k} \mathbf{v}_{(\ell)}, \quad 1\leq \ell \leq \prod_{i=1}^d n_i,
$$ 
where $1 \leq i_j \leq \alpha_j$ for $j=1,\ldots,d-1$, $\mathbf{v} \in \C^{\prod_{i=1}^d n_i}$, and $\mathbf{v}_{(\ell)}$ denotes the $\ell$th entry of $\mathbf{v}$. We use the structure of the eigenvectors to extract the coordinates $x_1^*,\ldots,x_{d-1}^*$ of the solutions. 

Therefore, our generalization of the Dixon resultant suggests a complete algorithm to solve a generic PMEP. We construct the resultant according to the above definition, linearize and solve the resulting PEP with QZ, and then extract the $x_1,\ldots,x_{d-1}$ coordinates from the eigenvectors of the resultant. For a generic PMEP,  an outline of the algorithm is given below, with detailed explanation to follow.

\begin{algorithm}
\textbf{MultiPolyEig: Method to Solve Generic PMEPs.}

\hrule

\vspace{.1cm}

\textbf{Input:} PMEP in \cref{eq:MPEigForm}.
\begin{algorithmic}[1]
\State Construct the hidden variable tensor Dixon resultant $R(x_d)$.
\State Linearize $R(x_d)$ to $A - x_d B$.
\State Solve with $QZ$. Obtain pairs $x_d^*,V^*$ of eigenvalues and right eigenvectors of $R(x_d)$.
\State Extract solution coordinates $x_i^*$ as ratios of entries of $V^*$.
\State Perform a residual check to discard spurious solutions.
\end{algorithmic}
\textbf{Output:} All solutions $(x_1^*,\ldots,x_d^*)$ to \cref{eq:MPEigForm}.
\label{alg:mpe}
\end{algorithm}

While our method is designed to work in its simplest form for generic $d$-degree PMEPs, it is useful to illustrate our algorithm with an example that can be calculated by hand. Consider the following PMEP: 

\begin{equation} \label{eq:simpleex}
\begin{aligned}
P_1 &= \begin{bmatrix}
1 & 0 \\
0 & 1 \\
\end{bmatrix} x^2 +
\begin{bmatrix}
0 & 1 \\
2 & 0
\end{bmatrix}, \\
P_2 &= \begin{bmatrix}
0 & 1 \\
-1 & 0 \\
\end{bmatrix} xy +
\begin{bmatrix}
-1 & 0 \\
-1 & 1
\end{bmatrix}.
\end{aligned}
\end{equation}
We will illustrate our generic PMEP solver on this example. 

\subsection{Constructing the Dixon function}

Instead of directly constructing the numerator of the Dixon function, we have found it to be slightly faster to compute it at many values of $x_d$ and interpolate. The analysis of \cref{sec:theory} gives a bound on the degree of $f_{\text{Dixon}}$ in $x_d$, so we construct $f_{\text{Dixon}}$ at a sufficient number of values of $x_d$ given its degree in $x_d$. For each value of $x_d$, we compute the numerator of $f_{\text{Dixon}}$ by directly expanding the block determinant formula.

We then use an observation of \cite{nakatsukasa2015bezout}, originally due to \cite{nakatsukasa2017vs}, that the division by $\prod_{i=1}^{d-1} (s_i-t_i)$ corresponds to solving a Sylvester equation, to obtain $f_{\text{Dixon}}$. This requires one iteration of a modified Bartels--Stewart algorithm for each $i=1,\ldots,d-1$. Subsequently, we unfold the tensor in \cref{eq:Dixonexpansion} into a matrix. Finally, we interpolate at each matrix entry to obtain the tensor Dixon matrix polynomial $R(x_d)$.\footnote{In practice, we construct $f_{\text{Dixon}}$ in the basis of Chebyshev polynomials of the first kind, and so we use Chebyshev nodes for interpolation. The interpolation is done efficiently using the fast Fourier transform, which we modify to mimic a discrete Chebyshev transform.}


We can calculate the entire first step directly for the system in \cref{eq:simpleex}. In this case, the numerator of $f_{\text{Dixon}}$ is
\begingroup
\setlength\arraycolsep{2pt}
\begin{align*}
&P_1(s,y) \otimes  P_2(t,y) -  P_1(t,y) \otimes P_2(s,y) = \\ &
\begin{bmatrix}
0 & s^2t-st^2 & 0 & t-s \\
st^2-s^2t  & 0 & s-t & 0 \\
0 & 2(t-s) & 0 & s^2t-st^2  \\
2(s-t) & 0 & st^2-s^2t  & 0
\end{bmatrix}
\hspace{-.1cm} y \hspace{-.1cm} + \hspace{-.2cm}
\begin{bmatrix}
t^2-s^2 & 0 & 0 & 0 \\
t^2-s^2  & s^2-t^2 & 0 & 0 \\
0 & 0 & t^2-s^2 & 0 \\
0 & 0 & t^2-s^2 & s^2-t^2
\end{bmatrix}.
\end{align*}
\endgroup
Then, dividing by $s-t$ gives
$$
f_{\text{Dixon}} = 
\begin{bmatrix}
0 & st & 0 & -1 \\
-st  & 0 & 1 & 0 \\
0 & -2 & 0 & st  \\
2 & 0 & -st  & 0
\end{bmatrix}
y +
\begin{bmatrix}
-s-t & 0 & 0 & 0 \\
-s-t & s+t & 0 & 0 \\
0 & 0 & -s-t & 0 \\
0 & 0 & -s-t & s+t
\end{bmatrix}.
$$
The resultant $R(y)$ is the matrix expansion of the coefficients of this function in $s$ and $t$, which is
\setlength{\tabcolsep}{0pt}
$$
\begin{array}{@{}c@{}c}
& \hspace{-.7cm} \begin{array}{*{9}{c}} \bold{1} &&&&&&&& \bold{s} \end{array} \\
\begin{array}{*{1}{c}} \bold{1} \\ \\ \\ \\ \bold{t} \end{array}  &
\begin{bmatrix}
0 & 0 & 0 & -1 & 0 & 0 & 0 & 0 \\
0 & 0 & 1 & 0 & 0 & 0 & 0 & 0 \\
0 & -2 & 0 & 0 & 0 & 0 & 0 & 0 \\
2 & 0 & 0 & 0 & 0 & 0 & 0 & 0 \\
0 & 0 & 0 & 0 & 0 & 1 & 0 & 0 \\
0 & 0 & 0 & 0 & -1 & 0 & 0 & 0 \\
0 & 0 & 0 & 0 & 0 & 0 & 0 & 1 \\
0 & 0 & 0 & 0 & 0 & 0 & -1 & 0 \\
\end{bmatrix}y\,\,\,\,+  \
\end{array}
\begin{array}{*{1}{c}}
\begin{array}{*{8}{c}} \bold{1} &&&&&&& \bold{s} \end{array} \\
\begin{bmatrix}
0 & 0 & 0 & 0 & -1 & 0 & 0 & 0 \\
0 & 0 & 0 & 0 & -1 & 1 & 0 & 0 \\
0 & 0 & 0 & 0 & 0 & 0 & -1 & 0 \\
0 & 0 & 0 & 0 & 0 & 0 & -1 & 1 \\
-1 & 0 & 0 & 0 & 0 & 0 & 0 & 0 \\
-1 & 1 & 0 & 0 & 0 & 0 & 0 & 0 \\
0 & 0 & -1 & 0 & 0 & 0 & 0 & 0 \\
0 & 0 & -1 & 1 & 0 & 0 & 0 & 0 \\
\end{bmatrix}
\end{array} \hspace{-.2cm}.
$$

In general, we obtain a matrix polynomial $R(x_d)$ that is not necessarily linear, and we linearize as detailed in \cref{subsec:lin}. Then we solve using $QZ$ and extract eigenvalue-eigenvector pairs $x_d^*,V^*$ of $R(x_d)$.

\subsection{Extracting Solution Coordinates for Generic PMEPs}
We know that the eigenvectors of $R(x_d)$ have size $\prod_{i=1}^d n_i$ blocks of the form $\prod_{k=1}^{d-1} (x_k^*)^{i_k} \mathbf{v}$. We can visualize this structure as
$$
V^* = \begin{bmatrix}
\mathbf{v} \\
x_1^* \mathbf{v} \\
(x_1^*)^2 \mathbf{v} \\
\vdots \\
x_2^* \mathbf{v} \\
\vdots
\end{bmatrix}.
$$
Consequently, we can obtain the coordinates $x_1^*,\ldots,x_{d-1}^*$ of the solutions by dividing any entry in the block $x_i \mathbf{v}$ by the corresponding entry in the block $\mathbf{v}$. In practice, we average over many such quotients for greater stability.

We demonstrate this on a single eigenvalue-eigenvector pair from our running example. We obtain eigenvalue $y^* \approx -1.3606$ and corresponding eigenvector 
$$
\begin{bmatrix}
\\ \\
\color{green!40!black!70}{\mathbf{v}} \\
\\ \\
\color{red!50!black!70}{x^* \mathbf{v}} \\
\\ \\
\end{bmatrix} \approx 
\begin{bmatrix} \color{green!40!black!70}{-0.7071} \\  \color{green!40!black!70}{0.4370} \\  \color{green!40!black!70}{1.0000} \\  \color{green!40!black!70}{-0.6180} \\  \color{red!50!black!70}{-0.5946} \\ \color{red!50!black!70}{0.3675} \\ \color{red!50!black!70}{0.8409} \\ \color{red!50!black!70}{-0.5197} \end{bmatrix}.
$$

The block colored red has the form $x^* \mathbf{v}$ and the block colored green has the form $\mathbf{v}$, where $(x^*,y^*)$ is an eigenvalue of the original PMEP in \cref{eq:simpleex}. Therefore, we can find $x^*$ by dividing any entry in the red block by the corresponding entry in the green block. For numerical stability, we divide all entries and average to obtain $x^* \approx 0.8409$. Completing this process for all of the eigenvalue-eigenvector pairs of $R(y)$ yields the full set of $8$ solutions. 


\subsection{Residual Check}

The algorithm can find spurious or extremely inaccurate solutions, which is especially an issue as the size of the PMEP increases, even for generic problems. This is partly due to potential numerical instability for large problems, but there can also be spurious symbolic roots. Even without numerical error, our resultant can have eigenvalues that are not coordinates of solutions to the original PMEP. In the language of algebraic geometry, this is because what we calculate is not a precise analogue of the true resultant, which vanishes only at coordinates of solutions, but instead, a multiple of the resultant, and the extraneous factor can have roots that do not give solutions to the original problem.\footnote{There are more details about this issue in \cite[Chapt. ~3]{cox2005ag} and \cite{saxenathesis} for the scalar polynomial resultant. Moreover, \cite[Chapt. ~3]{cox2005ag} gives a precise definition of the resultant that makes it unique up to a scalar multiple. Then, all other resultant-like constructions are multiples of this uniquely defined resultant.} 

Therefore, we implement a residual check to ensure the quality of all returned solutions. We find the maximum of the smallest singular value of $P_i(\mathbf{x}^*)$ for $i =1,\ldots,d$. For the roots found above, the maximum residual over all the roots is $10^{-15}$.

\section{Theory for the Tensor Dixon Resultant for Generic PMEPs}
\label{sec:theory}

We now prove several properties of the tensor Dixon resultant for generic systems (defined in \cref{subsec:gen}). Most importantly, we demonstrate for any PMEP in \cref{eq:MPEigForm} that the resultant polynomial that we construct has eigenvalues corresponding to all solutions of the PMEP. For generic PMEPs, we also prove that the resulting matrix polynomial is nonsingular. The analysis of this section is based on analysis of the scalar Dixon resultant in \cite{saxenathesis} and the hidden variable Dixon resultant in \cite{noferini2016instability}. The first proposition we need proves that matrix-vector multiplication by the resultant is equivalent to evaluating the Dixon function.

\begin{prop}
\label{prop:tensorexp}
Let $P = \{P_i : 1 \leq i \leq d \}$ be a PMEP as in \cref{eq:MPEigForm}, expressed in maximal degree $\tau_1,\ldots,\tau_d$ form as in \cref{eq:maxdegsystem}, where the matrix coefficients of $P_i$ are $n_i \times n_i$ for $n_1,\ldots,n_d \in \mathbb{Z}_{> 0}$. Let $N =  \prod_{i=1}^d n_i$, $\alpha_i =  i\tau_i - 1$, and $\beta_i = (d-i)\tau_i-1$. Let $U$ be a $\alpha_1 \times \cdots \times \alpha_{d-1} \times N$ tensor with
$$
U_{i_1,\ldots,i_{d-1},\ell} = \prod_{k=1}^{d-1} (x_k^*)^{i_k} \mathbf{v}_{(\ell)}, \quad 1 \leq i_k \leq \alpha_k, \quad 1 \leq \ell \leq N,
$$
with $\mathbf{v} \in \C^{N}$, $x_k^* \in \C, 1 \leq k \leq d-1$. Let $V = \text{vec}(U)$ be the corresponding multivariate block Vandermonde vector. Let $R(x_d)$ be the hidden variable tensor Dixon resultant for $P$, and $f_{\text{Dixon}}$ the associated Dixon function. Let $B$ be a $\beta_1 \times \cdots \times \beta_{d-1} \times N \times N$ tensor with
$$
B_{i_1,\ldots,i_{d-1},\ell_1,\ell_2} = \prod_{k=1}^{d-1} (t_k)^{i_k} (I_{N})_{(\ell_1,\ell_2)}, \quad 1 \leq i_k \leq \beta_k, \quad1 \leq \ell_i \leq N,
$$
with $I_N$ the $N \times N$ identity matrix and variables $t_k, 1 \leq k \leq d-1$. Let $W$ be the matrix with $N$ columns given by unfolding $B$ along the first $d-1$ coordinates and let $x_d^* \in \C$. Then
\begin{equation}
	W R(x_d^*) V = f_{\text{Dixon}}(x_1^*,\ldots,x_{d-1}^*,t_1,\ldots,t_{d-1},x_d^*) \mathbf{v}.
\end{equation}
\end{prop}

\begin{proof}
	Both sides are equivalent to the matrix polynomial
	$$
	\left( \sum_{i_1=0}^{\alpha_1} \cdots \sum_{i_{d-1}=0}^{\alpha_{d-1}} \sum_{j_1=0}^{\beta_1} \cdots \sum_{j_{d-1}=0}^{\beta_{d-1}} A_{i_1, \ldots, i_{d-1}, j_1, \ldots, j_{d-1}}(x_d^*) \prod_{k=1}^{d-1} (x_k^*)^{i_k} \prod_{k=1}^{d-1} t_k^{j_k} \right) \mathbf{v},
	$$
	where $A$ is as in \cref{eq:Dixonexpansion}.
\end{proof}

Now we can prove the result that allows our algorithm to find solutions to PMEPs. In particular, we show that the resultant $R(x_d)$ has eigenvalues that are the $d$th coordinates of solutions to the PMEP in \cref{eq:MPEigForm} and that the corresponding eigenvectors have multivariate block Vandermonde structure.

\begin{prop}
\label{thm:dixonev}
With the same setup as \cref{prop:tensorexp},  suppose that $x_1^*,\ldots,x_d^*$ and $\mathbf{v}_1,\ldots,\mathbf{v}_d$ are the solutions to the PMEP and let $\mathbf{v} = \mathbf{v}_1 \otimes \cdots \otimes \mathbf{v}_d$.  Let $U$ be an $\alpha_1 \times \cdots \times \alpha_{d-1} \times N$ tensor with
$$
U_{i_1,\ldots,i_{d-1},\ell} = \prod_{k=1}^{d-1} (x_k^*)^{i_k} \mathbf{v}_{(\ell)}, \quad 1 \leq i_k \leq \alpha_k, \quad 1 \leq \ell \leq N,
$$
with $\mathbf{v}_{(\ell)}$ the $\ell$th entry of $\mathbf{v}$. Then $x_d^*$ is an eigenvalue of the hidden variable tensor Dixon resultant $R(x_d)$ with eigenvector $V = \text{vec}(U)$.
\end{prop}

\begin{proof}
	By \cref{prop:tensorexp}, matrix-vector products with $R(x_d)$ are equivalent to evaluating the Dixon function. By construction, $f_{\text{Dixon}}(x_1^*,\ldots,x_{d-1}^*,t_1,\ldots,t_{d-1},x_d^*)$ is singular with right eigenvector $\mathbf{v}$, so 
	\begin{equation} \label{eq:fullrank}
	W R(x_d^*) V = f_{\text{Dixon}}(x_1^*,\ldots,x_{d-1}^*,t_1,\ldots,t_{d-1},x_d^*) \mathbf{v} = 0,
	\end{equation}
where $W$ is a block matrix consisting of copies of the identity multiplied by independent basis elements as in \cref{prop:tensorexp}. In particular, $W R(x_d^*) V$ is in general a matrix polynomial in variables $t_1,\ldots,t_d$, with coefficients given by the blocks of $R(x_d^*) V$. \Cref{eq:fullrank} implies that this matrix polynomial is identically $0$, so all its coefficients, which are the blocks of $R(x_d^*) V$, are $0$, so $R(x_d^*) V = 0$, as desired. 
\end{proof}


We now know that a hidden variable tensor Dixon method constructs a PEP with eigenvalue-eigenvector pairs that give the coordinates of the solutions to \cref{eq:MPEigForm}.  Unfortunately,  the constructed PEP can be singular (see~\cref{sec:finetuning}),  which requires additional computational work to circumvent. Therefore, examining a case in which the PEP is always nonsingular is useful.

Recall that we say that a PMEP is generic $d$-degree if it is of the form in \cref{eq:maxdegsystem} with coefficients that are independent variables. For $d \leq 3$, we prove that the resultant of a generic $d$-degree system is nonsingular.

\begin{thm} \label{thm:nonsing}
Let $d \leq 3$. Suppose that $R(x_d)$ is a hidden variable tensor Dixon resultant of a generic $d$-degree PMEP as in \cref{eq:maxdegsystem}. Then, $R(x_d)$ is nonsingular.
\end{thm}

\begin{proof}
The hidden variable tensor Dixon resultant $R(x_d)$ is nonsingular for generic systems if its determinant is a nonzero polynomial function of the coefficients of the PMEP. To prove that a polynomial is nonzero, finding any specialization of the coefficients that makes it nonzero is sufficient. Therefore, for each $d$-degree, it is sufficient to find an instance of a $d$-degree system for which $R(x_d)$ is nonsingular; i.e., its determinant is a nonzero polynomial function of $x_d$. Similarly to the proof of \cite[Theorem 2.6.2]{saxenathesis}, we choose the system 
\begin{equation} \label{eq:nonsingsystem}
	P_j =  \left( \prod_{i=1}^{d-1} (I_{n_j} x_i^{\tau_i} - A_j) \right) x_d^{\tau_d}, \quad 1 \leq j \leq d,
\end{equation}
with $A_j \in \C^{n_j \times n_j}$ indeterminates.
By examining the Dixon matrix, it is clear that
\begin{equation} \label{eq:dixexp}
f_{\text{Dixon}} = c \left( \prod_{i=1}^{d-1} \frac{s_i^{\tau_i} - t_i^{\tau_i}}{s_i-t_i} \prod_{j=1}^{i-1} \left( s_i^{\tau_i} - t_j^{\tau_j} \right) \right) x_d^{d \cdot \tau_d},
\end{equation}
with $
c = \prod_{i=1}^{d} \prod_{j=i+1}^{d} (B_i-B_j),
$
and $B_i = I_{n_1} \otimes \cdots \otimes I_{n_{i-1}} \otimes A_i \otimes I_{n_{i+1}} \otimes \cdots \otimes I_{n_d}$. In particular, this follows from the fact that the Dixon function must be divisible by each of the linear factors in \cref{eq:dixexp} and its degree bounded above by $\alpha_i =  i\tau_i - 1$ in $s_i$ and by $\beta_i = (d-i)\tau_i-1$ in $t_i$. 

We now rearrange and expand
\begin{equation} \label{eq:gensystemexp}
f_{\text{Dixon}} = c \left( \prod_{i=1}^{d-1} s_i^{\tau_i-1} + \cdots + t_i^{\tau_i-1} \right) \left( \prod_{i=1}^{d-1} \prod_{j=1}^{i-1} \left( s_i^{\tau_i} - t_j^{\tau_j} \right) \right) x_d^{d \cdot \tau_d}.
\end{equation}
The left-hand product has $\prod_{i=1}^{d-1} \tau_i$ distinct terms, each of which has a unique exponent in both $s$ and $t$; the right-hand product has $(d-1)!$ monomials in both $t$ and $s$. This gives a Dixon resultant of size $\left( \prod_{i=1}^d {n_i} \right) \left( (d-1)! \prod_{i=1}^{d-1} \tau_i \right)$. 

For $d=1,2$, the right-hand product is trivial. for $d=3$, the expansion is just $s_2^{\tau_2}-t_1^{\tau_1}$, Therefore for $d \leq 3$, $f_{\text{Dixon}}$ has the same property as the left hand expansion; each monomial is unique in both $t$ and $s$. Therefore the expansion of the resultant is block diagonal with diagonal blocks of size $\left( \prod_{i=1}^d {n_i} \right)$ all equal to $\pm c x_d^{d \cdot \tau_d}$, which implies that it is nonsingular.
\end{proof}

Many practical applications are in $d \leq 3$, though we believe that the above theorem holds for all $d$. By expanding the right-hand products of \cref{eq:gensystemexp}, we have symbolically verified that for $4\leq d \leq 10$,  there is an order on the monomials in the expansion that gives a Dixon resultant which is block upper triangular with diagonal blocks of size $\left( \prod_{i=1}^d {n_i} \right)$ all equal to $\pm c x_d^{d \cdot \tau_d}$. This gives us a computer proof of~\cref{thm:nonsing} for $4\leq d\leq 10$.  We have been unable to get a proof for any $d$; however, we strongly suspect that the theorem holds for any $d$.\footnote{In the proof of \cite[Theorem 2.6.2]{saxenathesis}, it is claimed that a similar system in the scalar case gives a Dixon matrix with determinant $c^{d! \prod_{i=1}^d \tau_i}$ for any $d$,  but we have been unable to verify the claim.  We believe that verifying this claim would be a crucial first step to proving~\cref{thm:nonsing} for any $d$. }

The practical interpretation of the previous theorem is that for randomly generated maximal degree systems, $R(x_d)$ is nonsingular with probability $1$. In practice, many systems are structured in ways that violate this assumption. In particular, total degree systems are common, as well as systems in which the matrix coefficients are themselves singular.

\section{Fine Tuning the Algorithm}
\label{sec:finetuning}

As shown in~\cref{sec:theory},  for randomly generated maximal-degree systems the generic algorithm from \cref{sec:outline} works with probability $1$.  In particular, the resulting eigenvalue problem is nonsingular with probability $1$. For practical applications, attention must be paid to the possibly singular polynomial eigenvalue problem that can result from the tensor resultant construction.
 
\subsection{Singular Polynomial Eigenvalue Problems}
\label{subsec:singgep}

If the assumptions of \cref{thm:nonsing} are violated, then the resultant pencil may be singular. Consider the system
\begin{equation} \label{eq:simplesingex}
\begin{aligned}
P_1 &= \begin{bmatrix}
0 & 1 \\
0 & 0 \\
\end{bmatrix} x^2 +
\begin{bmatrix}
0 & 1 \\
2 & 0
\end{bmatrix}, \\
P_2 &= \begin{bmatrix}
0 & 1 \\
0 & 0 \\
\end{bmatrix} xy +
\begin{bmatrix}
-1 & 0 \\
-1 & 1
\end{bmatrix}.
\end{aligned}
\end{equation}
By hiding $y$,  the Dixon function is
\begin{equation}
 \begin{bmatrix}
0 &  0 & 0 & st \\
0 &  0 & 0 & 0\\
0 & -2 & 0 & 0 \\
0 & 0 &  0 & 0 
  \end{bmatrix} y 
  +  \begin{bmatrix}
0 &  0 & -s-t & 0 \\
0 &  0 & -s-t & s+t \\
0 & 0 & 0 & 0 \\
0 & 0 &  0 & 0 
      \end{bmatrix},
\end{equation}
and our resultant is
\begin{equation} \label{eq:singexres}
 \begin{bmatrix}
 0 &  0 & 0 & 0 & 0 & 0 & 0 & 0 \\
0 &  0 & 0 & 0 & 0 & 0 & 0 & 0 \\
0 & -2 & 0 & 0 & 0 & 0 & 0 & 0\\
0 & 0 &  0 & 0 & 0 & 0 & 0 & 0 \\
0 & 0 & 0 & 0 & 0 &  0 & 0 & 1 \\
0 & 0 & 0 & 0 & 0 &  0 & 0 & 0\\
0 & 0 & 0 & 0 & 0 & 0 & 0 & 0 \\
0 & 0 & 0 & 0 & 0 & 0 & 0 & 0 
 \end{bmatrix} y +  \begin{bmatrix}
 0 &  0 & 0 & 0 & 0 & 0 & -1 & 0 \\
0 &  0 & 0 & 0 & 0 & 0 & -1 & 1 \\
0 & 0 & 0 & 0 & 0 & 0 & 0 & 0\\
0 & 0 &  0 & 0 & 0 & 0 & 0 & 0 \\
0 & 0 & -1 & 0 & 0 &  0 & 0 & 0 \\
0 & 0 & -1 & 1 & 0 &  0 & 0 & 0\\
0 & 0 & 0 & 0 & 0 & 0 & 0 & 0 \\
0 & 0 & 0 & 0 & 0 & 0 & 0 & 0 
 \end{bmatrix}.
\end{equation}
 The resultant has $3$ zero rows and columns, so for any $y$ it is of rank no more than $5$, while we would usually expect rank $8$ for generic/random $y$.
 
Unfortunately, this issue is not limited to contrived examples. As seen in \cref{subsec:aef}, singular resultant pencils arise often in practice. To resolve this difficulty, we use the projection method proposed in \cite{bor2024sing}. The method projects a singular pencil $R(x_d)$ orthogonally onto a pencil $U R(x_d) V^*$ from which we can extract a smaller nonsingular pencil. After linearization, the standard MATLAB QZ algorithm can be applied. This allows us to reliably find eigenvalues and eigenvectors of singular PEPs such as the one in \cref{eq:singexres} or those arising from practical examples in \cref{sec:num}.

As discussed in detail in \cite{bor2024sing}, another obvious approach is to linearize the singular pencil first and then apply a method for singular GEPs, such as from \cite{hochstenbach2019sing,hochstenbach2023sing}. Our experiments agree with the suggestion of \cite{bor2024sing} that these methods do not perform as well as projecting the pencil. In addition, we have found in our numerical experiments that the projection method outperforms the other methods described in \cite{bor2024sing}.

\subsection{Extracting Solutions From Eigenvectors for Singular Problems}
\label{subsec:evs}

Aside from introducing difficulties in solving the resultant pencil eigenvalue problem, the singularity also corrupts the information in the eigenvectors we use to extract roots. One extremely costly option is to plug in the $x_d$ coordinate and reduce the original PMEP to a large set of smaller problems, which can be solved by repeated application of the hidden variable resultant method. We leave this only as a last resort as it involves constructing many new hidden variable resultants and is, in our observation, often unacceptably slow. 

Instead, we find in practice that the null space of the pencil $R(x_d)$ for practical problems is often sparse enough to extract the coordinates $x_1^*,\ldots,x_{d-1}^*$ of the eigenvalues. We call the null space of $R(x_d)$, where $x_d$ is a variable, the generic null space. We estimate the sparsity pattern of the generic null space by evaluating $R(x_d)$ at a random complex number. Then we observe that, due to the block structure of the eigenvectors, the $x_i$ coordinate can be found by dividing any entry in the block associated to $x_i$ by the corresponding entry in the block associated to $1$. Some of these entries are corrupted by the generic null space; the eigenvectors now have the same underlying block Vandermonde structure plus some vector in the generic null space. In practice, we can choose entries that are sufficiently independent of the null space; we filter out entries where the effect of the null space is above some small numerical threshold, usually around $~10^{-13}$, and calculate from those remaining.

For better numerical accuracy, we also average among the ratios where the divisor is among the largest. The proportion of the entries used is a parameter that can be tuned by the user for optimal performance; in our experience, the optimal parameter varies somewhat among different applications.


\subsection{Extracting Solutions From Eigenvectors for Linear Problems}
\label{subsec:linear}

An additional difficulty occurs for PMEPs of maximal degree $\tau_1,\ldots,\tau_d$, where one or more of $\tau_1,\ldots,\tau_d$ is $1$. If $\tau_i = 1$ for $i < d$, then there is no block of the eigenvectors corresponding to any power of $x_i$ other than $x_i^0 = 1$, so it is impossible to extract the $i$th coordinate of the solutions by taking ratios of eigenvector entries. If possible, we avoid this case by swapping the coordinate order so that $\tau_d = 1$ but $\tau_i > 1$ for $i < d$. For some problems (i.e. when $\tau_i = \tau_j = 1$ for $i \neq j$), this is not possible. In these cases, we plug in the solutions that can be extracted from the eigenvectors one at a time, and then solve the smaller subproblems that result.

\subsection{Repeated Eigenvalues}

Many practical systems have multiple solutions that share an $x_d$ coordinate, which causes the resultant $R(x_d)$ to have a multiple eigenvalue. The space of eigenvectors associated to a multiple eigenvalue $x_d^*$ has dimension $> 1$, making it impossible to extract the coordinates $x_1^*,\ldots,x_{d-1}^*$ of the solutions. We can avoid this by making a random orthogonal change of coordinates $\begin{bmatrix} x_1' & \cdots &  x_d' \end{bmatrix}^{\top} = Q \begin{bmatrix} x_1 & \cdots &  x_d \end{bmatrix}^{\top}$, where $Q$ is $d \times d$ and orthogonal; this avoids repeated eigenvalue coordinates with probability one. We use this transformation for both of the practical experiments in \cref{sec:num}.

\section{Numerical Experiments}
\label{sec:num}

The primary use of the general algorithm we have constructed is to circumvent the time-consuming design of custom linearizations for new and potentially larger PMEPs that have not yet been approached. For previously studied problems, successful case-by-case methods already exist. We would not expect our general method to outcompete a refined method tailored to a particular problem on all metrics. However, PMEPs from \cite{pons2017aef,pons2018aef} and from \cite{gravenkamp2025leakywaves} still serve as useful benchmarks for our work and key tests for some of the numerical fine-tuning from \cref{sec:finetuning}. We compare our method to the pre-existing custom method for each problem. MATLAB code for all practical experiments is given in \cite{graf2025pmep}.

\subsection{Aeroelastic Flutter}
\label{subsec:aef}

A quadratic two parameter eigenvalue problem results from analyzing aeroelastic flutter \cite{pons2017aef,pons2018aef}. One model is \cite[eq. 10]{pons2018aef}
\begin{equation} \label{eq:aef}
	((M_0 + G_0) + G_1 \tau + G_2 \tau^2 -K_0 \Lambda) \mathbf{x} = 0,
\end{equation}
with $2 \times 2$ matrices given in \cite{pons2018aef}. A stability analysis of this model aims to find real solutions, so the authors construct \cite[eq. 12]{pons2018aef}
\begin{equation} \label{eq:AEF}
\begin{aligned}
	((M_0 + G_0) + G_1 \tau + G_2 \tau^2 -K_0 \Lambda) \mathbf{x} &= 0, \\
    ((\overline{M_0} + \overline{G_0}) + \overline{G_1} \tau + \overline{G_2} \tau^2 - \overline{K_0} \Lambda) \overline{\mathbf{x}} &= 0,
\end{aligned}
\end{equation}
which has a solution only for the real solutions of \cref{eq:aef}.
The authors of \cite{pons2017aef,pons2018aef} use several ad hoc methods as well as a two parameter quadratic solver from \cite{bor2010quadeig,hochstenbach2012quadeig} to solve \cref{eq:AEF}. In each case, the particular structure of this problem is exploited to construct a linear two parameter eigenvalue problem with the same solutions. We solve \cref{eq:AEF} with our tensor Dixon resultant for comparison.

As is frequently the case in practice, the structure of \cref{eq:AEF} causes the resultant pencil $R(\Lambda)$ to be singular. It is quadratic in $\Lambda$, with $8 \times 8$ matrix coefficients, but has a zero column, so it is rank $\leq 7$ for any value of $\Lambda$. The projection method of \cite{bor2024sing} allows us to extract a $7 \times 7$ quadratic pencil that is nonsingular before linearizing and solving with QZ. As explained in \cref{subsec:linear}, by hiding $\Lambda$ instead of $\tau$, we preserve enough Vandermonde structure in the eigenvectors to extract the $\tau$-coordinates of the solutions. We use the method from \cref{subsec:evs} to avoid the noise from the generic null space of $R(\Lambda)$ and accurately calculate the $\tau$-coordinates.

The problem in \cref{eq:AEF} has four solutions, which are illustrated in \cite[fig. 3]{pons2018aef}. We compare the solutions from our method with the results of quad\_twopareig from \cite{bor2010quadeig,hochstenbach2012quadeig}. The results in \cref{tab:AEF} illustrate that we obtain identical solutions up to $10^{-12}$.

\begin{table}
\resizebox{\textwidth}{!}{
\begin{tabular}{c}
$\tau$-coordinates \\
\hline
\begin{tabular}{cc}
MultiPolyEig & quad\_twopareig \\
\hline
$\color{green!40!black!70}-3.31759890823\color{red!50!black!70}7174 \color{green!40!black!70} + 0.0000000000000\color{red!50!black!70}14i$ & $\color{green!40!black!70}-3.31759890823\color{red!50!black!70}8001 \color{green!40!black!70}- 0.0000000000000\color{red!50!black!70}05i$ \\
 $\color{green!40!black!70}-0.912270188816\color{red!50!black!70}418 \color{green!40!black!70}- 0.0000000000000\color{red!50!black!70}35i$ & $\color{green!40!black!70}-0.912270188816\color{red!50!black!70}352 \color{green!40!black!70}+ 0.0000000000000\color{red!50!black!70}26i$ \\
 $\color{green!40!black!70}-0.91227018881640\color{red!50!black!70}4 \color{green!40!black!70}- 0.0000000000000\color{red!50!black!70}14i$ & $\color{green!40!black!70}-0.91227018881640\color{red!50!black!70}1 \color{green!40!black!70}- 0.0000000000000\color{red!50!black!70}00i$ \\
 $\color{green!40!black!70}-0.50086581799891\color{red!50!black!70}8 \color{green!40!black!70}- 0.00000000000000\color{red!50!black!70}1i$ & $\color{green!40!black!70}-0.50086581799891\color{red!50!black!70}7 \color{green!40!black!70}- 0.00000000000000\color{red!50!black!70}0i$
\end{tabular}
\\
 \hline
\end{tabular}}

\resizebox{\textwidth}{!}{
\begin{tabular}{c}
$\Lambda$-coordinates \\
\hline
\begin{tabular}{cc}
MultiPolEig & quad\_twopareig \\
\hline
 $\color{green!40!black!70}-0.00000000000\color{red!50!black!70}1507 \color{green!40!black!70}- 0.000000000000\color{red!50!black!70}393i$ & $\color{green!40!black!70}0.00000000000\color{red!50!black!70}0000 \color{green!40!black!70} + 0.000000000000\color{red!50!black!70}004i$ \\
 $\color{green!40!black!70}-4.137012225428\color{red!50!black!70}568 \color{green!40!black!70}+ 0.0000000000000\color{red!50!black!70}94i$ & $\color{green!40!black!70}-4.137012225428\color{red!50!black!70}681 \color{green!40!black!70}- 0.0000000000000\color{red!50!black!70}34i$ \\
 $\color{green!40!black!70}4.1370122254286\color{red!50!black!70}82 \color{green!40!black!70}- 0.0000000000000\color{red!50!black!70}34i$ & $\color{green!40!black!70}4.1370122254286\color{red!50!black!70}14 \color{green!40!black!70}+ 0.0000000000000\color{red!50!black!70}15i$ \\
 $\color{green!40!black!70}-0.00000000000000\color{red!50!black!70}1 \color{green!40!black!70}+ 0.00000000000000\color{red!50!black!70}2i$ & $\color{green!40!black!70}-0.00000000000000\color{red!50!black!70}2 \color{green!40!black!70}- 0.00000000000000\color{red!50!black!70}3i$
 \end{tabular}
 \\
 \hline
\end{tabular}}

\caption{Results of MultiPolyEig \cite{graf2025pmep} and quad\_twopareig \cite{bor2025mpe} applied to aeroelastic flutter problem. The coloring shows that our results agree with those of quad\_twopareig \cite{bor2025mpe} up to at least $10^{-12}$ for all solutions.}
\label{tab:AEF}

\end{table}

\subsection{Leaky Waves in Layered Structures}
\label{subsec:leakywaves}

\begin{figure}
\begin{center}
\includegraphics[scale=.15]{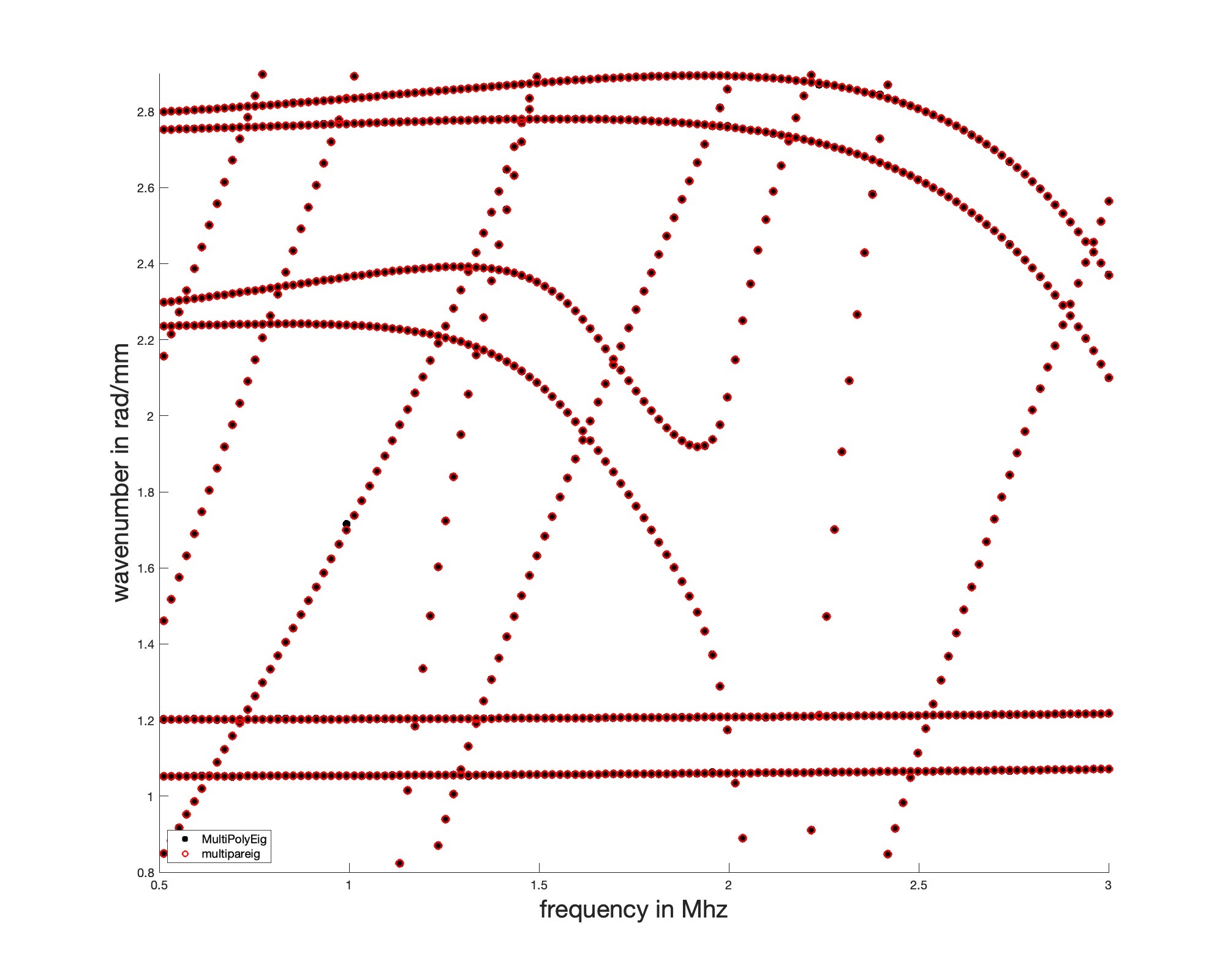}
\end{center}
\vspace{-1cm}
\caption{Performance of ad hoc linearization and multipareig  \cite{bor2025rji} versus MultiPolyEig \cite{graf2025pmep} on the system in \cref{eq:LeakyWaves}.}
\label{fig:LeakyWaves}
\end{figure}

In \cite{gravenkamp2025leakywaves}, the authors exploit multiparameter eigenvalue problems to compute wavenumbers of leaky waves in layered structures coupled to unbounded media. The initial model is given by \cite[eq. 8]{gravenkamp2025leakywaves}
\begin{equation}
\left( -k^2 E_0 + i k E_1 - E_2 + \omega^2 M + R(k) \right) \phi = 0,
\end{equation}
with $\phi$ the eigenvector, where $R$ has nonpolynomial dependence on $k$. Resolving the nonpolynomial terms gives a three parameter eigenvalue problem of the form
\begin{equation}
\label{eq:LeakyWaves}
\begin{aligned}
    \left( -E_2 + \omega^2 M + ik E_1 + i\eta_1 R_1 + i \eta_2 R_2 - k^2 E_0 \right) u &= 0, \\
    \left( \begin{bmatrix}
        0 & -\kappa_1^2 \\
        1 & 0
    \end{bmatrix} + i \eta_1 \begin{bmatrix}
        1 & 0 \\
        0 & 1
    \end{bmatrix} + k^2 \begin{bmatrix}
        0 & 1 \\
        0 & 0
    \end{bmatrix}\right) x_1 &= 0, \\
    \left( \begin{bmatrix}
        0 & -\kappa_2^2 \\
        1 & 0
    \end{bmatrix} + i \eta_2 \begin{bmatrix}
        1 & 0 \\
        0 & 1
    \end{bmatrix} + k^2 \begin{bmatrix}
        0 & 1 \\
        0 & 0
    \end{bmatrix} \right) x_2 &= 0,
\end{aligned}
\end{equation}
with the matrices in the first equation of size $25 \times 25$. \cite{bor2025rji} contains the necessary data for all the matrices involved. The objective is to calculate $k,\kappa_1,\kappa_2$ given values of $\eta_1,\eta_2,\omega$.

The authors of \cite{gravenkamp2025leakywaves} construct an ad hoc linearization \cite[eq. 36]{gravenkamp2025leakywaves} based on the particular structure of this problem, with code given in \cite{bor2025rji}. Our method can handle this problem directly. For direct comparison, we produce a version of \cite[fig. 8]{he2024rand} in which we compare the wavenumbers computed using our method to those computed using a linearization and multipareig. We compare the results in \cref{fig:LeakyWaves}, which graphs the computed wavenumbers against the input frequencies for $150$ different sampled frequencies. Our method finds almost the exact solutions even though the linearization from \cite{gravenkamp2025leakywaves} that is used in \cite{he2024rand} is highly specialized to the problem.

\section*{Acknowledgements}

E.G. was supported by NSF GRFP (DGE-2139899) and A.T. was supported by NSF CAREER (DMS-2045646).

\appendix

\section{Connections Between Methods for Multivariate Rootfinding and Eigenvalue Problems}
\label{app:conn}

In \cref{tab:evp}, we reference connections between methods for several multiparameter computational problems: multivariate linear systems solved by Cramer's rule, multivariate polynomial systems solved by the hidden variable Dixon resultant, the multiparameter eigenvalue problem solved by operator determinants, and the polynomial multiparameter eigenvalue problem solved by the hidden variable tensor Dixon resultant. We make these connections more explicit in \cref{fig:evpmethods}.

\begin{figure}
\begin{center}
\begin{tikzpicture}
\draw
	(0,0) node [fill = green!70!black!30,draw,double,rounded corners] {Cramer's Rule}
	(6,0) node [fill = red!70!green!70!black!30,draw,double,rounded corners] {Operator Determinants}
	(0,-3) node [fill = blue!70!green!70!black!30,draw,double,rounded corners] {Dixon Resultant}
    	(6,-3) node [fill = red!70!green!70!blue!70!black!30,draw,double,rounded corners] {Tensor Dixon Resultant};
	\draw
	[line width=1mm,-{Stealth[length=10mm, open]}]  (3.5,0) -> (1.7,0);
	\draw
	[line width=1mm,-{Stealth[length=10mm, open]}]  (3.5,-3) -> (1.7,-3);
	\draw
	[line width=1mm,-{Stealth[length=10mm, open]}]  (0,-2.5) -> (0,-.5);
	\draw
	[line width=1mm,-{Stealth[length=10mm, open]}]  (6,-2.5) -> (6,-.5);
	\draw
	[line width=1mm,-{Stealth[length=10mm, open]}]  (6,1) -> (0,1);
	\draw
	(3,1.7) node [fill = black!30,rounded corners] {Matrix to Scalar Coefficients};
	\draw
	[line width=1mm,-{Stealth[length=10mm, open]}]  (-2,-3) -> (-2,0);
	\draw
	(-3,-1.5) node [fill = black!30,rounded corners,rotate=90,text width=4cm] {High-Degree Polynomial to Linear};
\end{tikzpicture}
\caption{The connections between methods for four different computational algebra and geometry problems. }
\label{fig:evpmethods}
\end{center}
\end{figure}
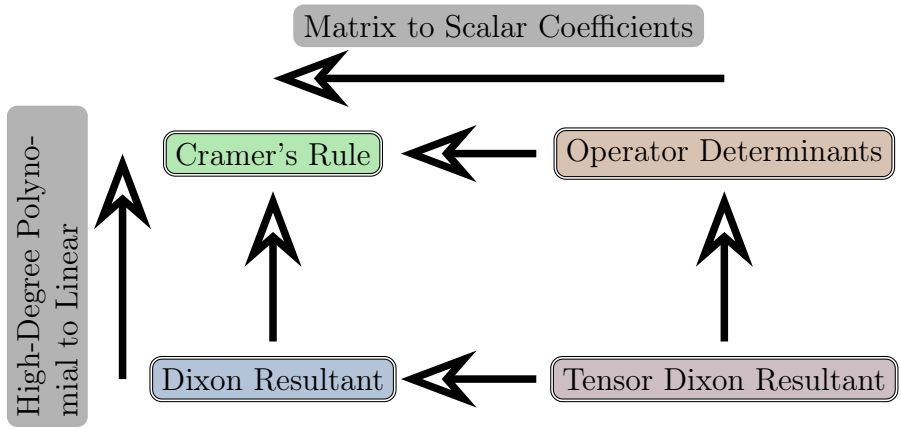

The arrows up and to the left indicate the specialization of our tensor Dixon resultant. Most of these connections are known or readily apparent from the definitions. The fact that operator determinants specialize in the scalar case to Cramer's rule is evident from the definitions and is well known \cite{atkinson1972multieig}. It is also clear from our definitions that the tensor Dixon resultant specializes to the Dixon resultant for scalar polynomials, as it is given by the same formula, with the block/Kronecker determinant that we defined for PMEPs specializing to the standard determinant (in particular the Kronecker product of scalars in $\C$ is the same as the standard product in $\C$). The fact that the Dixon resultant in the linear case is equivalent to Cramer's rule is proved in \cite[Theorem 3.4]{noferini2016instability}. We prove the analogous connection between the tensor Dixon resultant and operator determinants. The proof follows from the proof of \cite[Theorem 3.4]{noferini2016instability}.

\begin{prop}
	Let $W$ be a linear multiparameter eigenvalue problem (MEP) of the form:
	\begin{equation} \label{eq:MEP2}
	W_i(\mathbf{x}) \mathbf{v}_i = \left(V_{i0} - \sum_{j=1}^d x_j V_{ij} \right) \mathbf{v}_i = 0, \quad 1 \leq i \leq d,
	\end{equation}
	with $V_{ij} \in \mathbb{C}^{n_i \times n_i}$ for integers $n_1,\ldots,n_d$.  The GEP that results from applying the hidden variable tensor Dixon resultant method to $W$, hiding $x_d$, is the same as the GEP for $x_d$ that results from applying the operator determinants method to $W$.
\end{prop}

\begin{proof}
In the proof of \cite[Theorem 3.4]{noferini2016instability}, the authors consider a linear system $A \mathbf{x} = \mathbf{b}$. They let $A_d$ be the last column of $A$ and $B = A - A_d \mathbf{e_d}^{\top} + \mathbf{b} \mathbf{e_d}$, where $e_d$ is the $d$-th canonical basis vector. Then they prove that $f_{\text{Dixon}} = \text{det}(B) + \text{det}(A) x_d$.\footnote{The authors of \cite{noferini2016instability} refer to the Dixon function as the Cayley function.}

We can establish a correspondence between their argument and the tensor Dixon formulation as follows. The map that sends $V_{ij} \to A_{ij}$ for $j > 0$, $V_{i0} \to \mathbf{b}_i$, and sums of Kronecker products of matrices $V_{ij}$ to sums of products of elements of $A,\mathbf{b}$ is a bijection between the matrix coefficients $V_{ij}$ of the MEP in \cref{eq:MEP2} (and the corresponding combinations) and entries of $A,\mathbf{b}$ in the linear system $A \mathbf{x} = \mathbf{b}$ (and the corresponding combinations). By construction, this bijection commutes with taking Kronecker products on the left-hand side and products on the right-hand side and with sums on either side, so it commutes with the operation of constructing the Dixon function.

Therefore $f_{\text{Dixon}} = \text{det}(B) + \text{det}(A) x_d$ for the linear system implies that for the corresponding MEP
\begin{equation}
f_{\text{Dixon}} = \Delta_d -x_d \Delta_0,
\end{equation}
as in \cref{eq:opdetGEP}, which is the same as the GEP for $x_d$ that results from applying the operator determinants method to $W$.

%
\end{proof}

This analysis clearly applies to any of the variables, which makes it clear that repeating the hidden variable method for each variable would construct all of the GEPs in \cref{eq:opdetGEP}. This completes the connections in \cref{fig:evpmethods}.

\bibliographystyle{elsarticle-num} 
\bibliography{references.bib}

\end{document}